\documentclass{amsart}
\usepackage{amsmath}
\usepackage{amsfonts}
\usepackage{amssymb}
\usepackage{mathtools}
\usepackage{epsfig}
\usepackage{tikz}
\usepackage{url}
\usepackage{array}

\newtheorem{theorem}{Theorem}[section]
\newtheorem{prop}[theorem]{Proposition}
\newtheorem{lemma}[theorem]{Lemma}
\newtheorem{cor}[theorem]{Corollary}
\theoremstyle{definition}
\newtheorem{rem}[theorem]{Remark}
\newtheorem{defi}[theorem]{Definition}

\newcommand{\ra}{\rightarrow}

\newcommand{\IC }{\mathbb{C}}
\newcommand{\IR}{\mathbb{R}}

\newcommand{\IZ}{\mathbb{Z}}
\newcommand{\IQ}{\mathbb{Q}}

\newcommand{\coloneqq}{:=}
\DeclareMathOperator{\Pic}{Pic}
\DeclareMathOperator{\nef}{Nef}
\DeclareMathOperator{\movb}{Mov}
\DeclareMathOperator{\post}{Pos}

\DeclareMathOperator{\id}{id}

\DeclareMathOperator{\Hom}{Hom}
\DeclareMathOperator{\divi}{div}

\DeclareMathOperator{\aut}{Aut}
\DeclareMathOperator{\bir}{Bir}
\DeclareMathOperator{\transc}{Tr}
\DeclareMathOperator{\grass}{Grass}

\DeclarePairedDelimiter{\ceil}{\lceil}{\rceil}

\newcommand{\rk}{\mathrm{rk}\,}
\newcommand{\ns}{\mathrm{NS}}

\newcommand{\hsk}{S^{\left[n\right]}}
\newcommand{\hskn}{K3^{\left[n\right]}}

\makeatletter
\def\blfootnote{\xdef\@thefnmark{}\@footnotetext}

\begin{document}

\title[Automorphisms of Hilbert schemes of points on K3 surfaces]{Automorphisms of Hilbert schemes of points on a generic projective K3 surface}

\author{Alberto Cattaneo}

\address{Universit\`a degli Studi di Milano, Dipartimento di Matematica ``F. Enriques'', Via Cesare Saldini 50, 20133 Milano, Italy. \newline \indent Universit\'e de Poitiers, Laboratoire de Math\'ematiques et Applications, T\'el\'eport 2, Boulevard Marie et Pierre Curie, 86962 Futuroscope Chasseneuil, France.}

\email{alberto.cattaneo1@unimi.it; alberto.cattaneo@math.univ-poitiers.fr}

\maketitle


\begin{abstract}
We study automorphisms of the Hilbert scheme of $n$ points on a generic projective $K3$ surface $S$, for any $n \geq 2$. We show that $\aut(\hsk)$ is either trivial or generated by a non-symplectic involution and we determine numerical and divisorial conditions which allow us to distinguish between the two cases. As an application of these results we prove that, for any $n \geq 2$, there exist infinite values for the degree of $S$ such that $\hsk$ admits a non-natural involution. This provides a generalization of the results of \cite{bcnws} for $n=2$.
\end{abstract}

\section{Introduction} \label{sec: intro}

We consider a complex, algebraic $K3$ surface $S$ with $\Pic(S)=\IZ H$, where we assume that $H$ is an ample line bundle with $H^2 = 2t$, $t \geq 1$. Let $\hsk$ be the Hilbert scheme of $n$ points on $S$: it is a hyperk\"ahler manifold, arising as resolution of singularities of the $n$-th symmetric power $S^{(n)}$. The second cohomology group $H^2(\hsk, \IZ)$ admits a lattice structure with respect to the Beauville--Bogomolov--Fujiki quadratic form (see \cite{beauville} for more details). In particular, as lattices $H^2(\hsk, \IZ) \cong H^2(S, \IZ) \oplus \IZ \delta$, where $2\delta$ is the class of the exceptional divisor of the Hilbert-Chow morphism $\hsk \ra S^{(n)}$ and $\delta^2 = -2(n-1)$. The ample class $H$ on $S$ induces a line bundle $\widetilde{H}$ on $\hsk$, whose first Chern class we denote by $h$. Then, we can take $\left\{ h, -\delta \right\}$ as a basis for the N\'eron-Severi lattice $\ns ( S^{\left[ n\right]}) \subset H^2(\hsk, \IZ)$.

Boissi\`ere, Cattaneo, Nieper-Wisskirchen and Sarti \cite{bcnws} computed the automorphism group $\aut(S^{\left[2\right]})$ for the Hilbert scheme of two points on a $K3$ surface $S$ as above, showing that it is either trivial or generated by a non-symplectic involution, i.e.\ an involution which does not fix the closed two-form generating $H^{2,0}(S^{\left[2\right]})$. More recently, Debarre and Macr\`i \cite{debarre_macri} determined the birational automorphism group $\bir(S^{\left[2\right]})$, which coincides with $\aut(S^{\left[2\right]})$ except when $t=5$. The aim of this paper is to generalize the results of \cite{bcnws}, providing a description of $\aut(\hsk)$ for any $n \geq 3$. In particular, Oguiso proved that this group is always finite (\cite[Corollary 5.2]{oguiso}).

In Section \ref{sec: preliminaries} we briefly recall several classical results on Pell's equation and we summarize the description of Bayer--Macr\`i \cite{bayer_macri_mmp} for the cones of divisors on $\hsk$. In Section \ref{sec: isometries} we use $n$-very ample classes on the surface $S$ to construct ample classes on $\hsk$; this allows us to prove, in Proposition \ref{unique_isometry}, that if the action on $\ns(\hsk)$ of an automorphism $f \in \aut(\hsk)$  is non-trivial, then it is a specific isometry of order two determined by the extremal rays of the ample cone. 

In Section \ref{section: autom group} we show that, for any $n \geq 2$, either $\aut(\hsk) = \left\{ \id \right\}$ or $\aut(\hsk) = \left\{ \id, f \right\}$, with $f$ a non-symplectic involution. This means that, passing from $n=2$ to $n \geq 3$, no new possible structures for the automorphism group of $\hsk$ arise. By studying the action of $f$ on the nef cone of $\hsk$, we prove the following. 

\begin{prop}
Let $S$ be an algebraic $K3$ surface with $\Pic(S)=\IZ H$, $H^2 = 2t$.
\begin{itemize}
\item If $t=1$, then $\aut(\hsk) = \left\{ \id, \iota^{[n]} \right\}$, where $\iota^{[n]}$  is the natural involution induced by the involution $\iota$ which generates $\aut(S)$.
\item If $2 \leq t \leq 2n-3$, then $\aut(\hsk) = \left\{ \id \right\}$.
\end{itemize}
\end{prop}

In Section \ref{sec: invariant pol} we investigate more in detail the action of the non-symplectic involution on $H^2(\hsk, \IZ)$, looking at fixed ample classes. For $n=2$, by \cite[Theorem 1.1]{bcnws} $\aut(S^{\left[2\right]})$ is generated by an involution if and only if there exists an ample class $\nu \in \ns(S^{\left[2\right]})$ of square two. We provide the following generalization.

\begin{theorem} \label{thm_intro: divisorial conditions}
Let $S$ be an algebraic $K3$ surface with $\Pic(S)=\IZ H$, $H^2 = 2t$, $t \geq 1$. Then $\hsk$ admits a non-symplectic, non-natural involution if and only if there exists a primitive ample class $\nu \in \ns(\hsk)$ with either
\begin{itemize}
\item $\nu$ of square two, or
\item $\nu$ of square $2(n-1)$ and divisibility $n-1$ in $H^2(\hsk, \IZ)$.
\end{itemize}
\end{theorem}

We prove one of the two implications of the theorem in a more general setting: the existence of an ample class $\nu \in \ns(X)$ as in the statement guarantees the existence of a non-symplectic involution on any hyperk\"ahler manifold $X$ which is deformation equivalent to the Hilbert scheme of $n$ points on a $K3$ surface (so-called manifolds of $\hskn$-type; see Proposition \ref{prop: existence of involutions}). A more in-depth study of non-symplectic involutions on manifolds of $\hskn$-type will be presented in the upcoming paper \cite{CCC}, by Chiara Camere, the author and Andrea Cattaneo.

Finally, in Section \ref{sec: final numerical conditions} we achieve a purely numerical characterization for the existence of a non-trivial automorphism on $\hsk$, in the same spirit of \cite{bcnws}. This can be done by using the descriptions of Bayer--Macr\`i for the movable cone and the nef cone of the manifold.

\begin{theorem}
Let $S$ be an algebraic $K3$ surface with $\Pic(S)=\IZ H$, $H^2 = 2t$, $t \geq 2$ and $n \geq 2$. Let $(z,w)$ be the minimal solution of Pell's equation $X^2 - t(n-1)Y^2 = 1$ with $X \equiv \pm 1$ (mod $n-1$). Then $\aut(\hsk) \neq \left\{ \id \right\}$ if and only if 
\begin{enumerate}
\item[\textit{(i)}] $t(n-1)$ is not a square;
\item[\textit{(ii)}] if $n \neq 2$, the equation $(n-1)X^2 - tY^2 = 1$ has no integer solutions;
\item[\textit{(iii)}] for all integers $\rho, \alpha$ as follows:
\begin{enumerate}
\item $\rho = -1$ and $1 \leq \alpha \leq n-1$, or
\item $\rho = 0$ and $3 \leq \alpha \leq n-1$, or
\item $1 \leq \rho < \frac{n-1}{4}$ and $\max\left\{4\rho+1,\ceil[\bigg]{2\sqrt{\rho(n-1)}}  \right\} \leq \alpha \leq n-1$
\end{enumerate}
the minimal solution $(X,Y)$ of Pell's equation 
\[X^2 - 4t(n-1)Y^2 = \alpha^2 - 4\rho(n-1)\]
with $X \equiv \pm \alpha$ (mod $2(n-1)$), if it exists, is such that $\frac{Y}{X} \geq \frac{w}{2z}$;
\item[\textit{(iv)}] there exist integer solutions either for the equation $(n-1)X^2 - tY^2 = -1$ or for the equation $X^2 - t(n-1)Y^2 = -1$.
\end{enumerate}
If these conditions are satisfied, then $ \nef(\hsk) = \overline{\movb(\hsk)} = \langle h, zh-tw\delta \rangle$ and $\aut(\hsk)$ is generated by a non-symplectic, non-natural involution, whose action on $\ns(\hsk) = \IZ h \oplus \IZ (-\delta)$ is given by the matrix
\begin{equation*}
\begin{pmatrix}
z & -(n-1)w \\
tw & -z
\end{pmatrix}.
\end{equation*}
\end{theorem}

As an application of the theorem, we prove that for each $n \geq 2$ there are infinite values of $t$ for which there exists a non-natural involution $f \in \aut(\hsk)$.

\begin{prop}
Let $S$ be an algebraic $K3$ surface with $\Pic(S)=\IZ H$, $H^2 = 2t$ and assume $t = (n-1)k^2 + 1$ for $k, n$ positive integers, $n \geq 2$. If $k \geq \frac{n+3}{2}$, there exists a non-symplectic, non-natural involution $f \in \aut(\hsk)$, whose invariant lattice in $H^2(\hsk, \IZ)$ is generated by an element of square two.
\end{prop}

We point out that several results of Sections \ref{sec: isometries}, \ref{section: autom group} have been obtained independently by Olivier Debarre in his notes \cite{debarre} (in preparation), where the interested reader can also find some details on the group $\bir(\hsk)$.

\smallskip
\textbf{Acknowledgements:} I wish to thank Alessandra Sarti for proposing the problem and for her helpful guidance throughout the project. I am also grateful to Bert van Geemen, for carefully reading a first draft of the paper and for his precious remarks. This work has greatly benefited from discussions with Samuel Boissi\`ere, Chiara Camere and Andrea Cattaneo. A special thanks goes to Emanuele Macr\`i, for many useful explanations and suggestions.  

\section{Preliminaries} \label{sec: preliminaries}

\subsection{Pell's equation} \label{sec: pell}

A quick overview of the basic theory of Pell's equation can be found in \cite[\S 2.1]{bcnws}. In this section we only fix the notation and recall the properties that we will need for our purposes.

\begin{defi}
A (generalized) Pell's equation is a diophantine equation in two variables $X,Y \in \IZ$ of the form
\[X^2 - rY^2 = m \]
with $r \in \mathbb{N}$ and $m \in \IZ$.\\
A solution $(X,Y)$ of the equation is called \emph{positive} if $X > 0, Y > 0$. If solutions exist, the \emph{minimal} solution is the positive solution with minimal $X$.
\end{defi}

\begin{rem} \label{rem: pell slope}
If $m > 0$ (respectively, $m < 0$), the minimal solution of Pell's equation $X^2 - rY^2 = m$ is also the solution which minimizes (respectively, maximizes) the slope $\frac{Y}{X} = \sqrt{\frac{1}{r} - \frac{m}{rX^2}}$.
\end{rem}

Clearly, for any $r \in \mathbb{N}$ there exist solutions for Pell's equation $X^2 - rY^2 = 1$; in particular, if $r$ is a square the only solutions are $(\pm 1, 0)$. In the case where $r$ is the product of two non trivial integers and $m=1$, we have the following result.

\begin{lemma} \label{lemma: pell}
Let $s, q \in \mathbb{N}$, with $q \neq 1$. If the equation $sX^2 - qY^2 = -1$ admits integer solutions, let $(a,b)$ be the positive one with minimal $X$. Then the minimal solution of Pell's equation $X^2 - sqY^2 = 1$ is $(2sa^2 +1, 2ab)$.
\end{lemma}

\begin{proof}
A more general statement can be found in \cite{debarre}. For the case $s=1$ see also \cite[Lemma 2.1]{bcnws}.
\end{proof}

The integer solutions of $X^2 - rY^2 = m$ can be divided into equivalence classes. Two solutions $(x, y)$ and $(x',y')$ are equivalent if there exists a solution $(z,w)$ of Pell's equation $X^2 - rY^2 = 1$ such that
\begin{equation} \label{eq: pell solutions}
\begin{cases}
x' = zx + rwy \\
y' = wx + zy
\end{cases}
\end{equation}

\noindent (for more details, see \cite[Chapter XXXIII, \S 18]{chrystal}). We define the \emph{fundamental solution} in an equivalence class to be the solution with smallest non-negative $Y$; if there are two solutions with this property in the same equivalence class, they are of the form $(X,Y)$, $(-X,Y)$, with $X > 0$: in this case we consider $(X,Y)$ to be the fundamental solution. By applying (\ref{eq: pell solutions}) recursively, all solutions in an equivalence class can be reconstructed from the fundamental one, after computing the minimal solution of $X^2 - rY^2 = 1$.

\begin{rem} \label{rem: fundamental solutions}
If $m > 0$, let $(X,Y)$ be a fundamental solution of $X^2 - rY^2 = m$ and $(z,w)$ the minimal solution of $X^2 - rY^2 = 1$. Then, either $(X,Y)$ or $(-X,Y)$ belongs to the closed interval, on the hyperbola of equation $X^2 - rY^2 = m$, delimited by the points $(\sqrt{m}, 0)$ and $(\sqrt{m(z+1)/2}, \sqrt{m(z-1)/(2r)})$. Notice, moreover, that all solutions $(x,y)$ contained in this interval are such that $0 \leq \frac{y}{x} < \frac{w}{z}$.
\end{rem}

\subsection{Cones of divisors on $\hsk$} \label{sec: intro cones}

If $S$ is an algebraic $K3$ surface, the cohomology ring $H^*(S, \IZ) = H^0(S, \IZ) \oplus H^2(S, \IZ) \oplus H^4(S, \IZ)$ admits a lattice structure of rank $24$, called the \emph{Mukai lattice}, with pairing $\left( (r, l, s), (r', l', s') \right) \coloneqq l \cdot l' - rs' - r's$. As a lattice, $H^*(S, \IZ)$ is even and isometric to $U^{\oplus 4} \oplus E_8(-1)^{\oplus 2}$, where $U$ is the even, unimodular, hyperbolic lattice of rank two and $E_8(-1)$ is the negative definite root lattice defined by the Dynkin diagram $E_8$. Moreover, $H^*(S, \IZ)$ carries a weight-two Hodge structure, whose $(1,1)$-part is $H^*_{\textrm{alg}}(S, \IZ) \coloneqq H^0(S, \IZ) \oplus \ns(S) \oplus H^4(S, \IZ)$ (recall that $\ns(S) = H^{1,1}(S) \cap H^2(S, \IZ)$). For any $n \geq 2$, we define the algebraic vector $v = (1, 0, 1-n) \in H^*_{\textrm{alg}}(S, \IZ)$. There exists a canonical isomorphism $\theta^{-1}: H^2(\hsk, \IZ) \ra v^\perp \subset H^*(S, \IZ)$, such that $\theta(v^\perp \cap H^*_{\textrm{alg}}(S, \IZ)) = \ns(\hsk)$. In particular, if $\Pic(S) = \IZ H$ the basis $\left\{h, -\delta \right\}$ of $\ns(\hsk)$ defined in \S \ref{sec: intro} is realized as
\[ h = \theta(0,-H,0), \quad -\delta = \theta(1,0,n-1)\]
\noindent where (with an abuse of notation) we still denote by $H$ the first Chern class of the generator of $\Pic(S)$.

The following lemma provides a description of algebraic classes $a \in H^*_{\textrm{alg}}(S, \IZ)$ with given square and pairing with $v = (1, 0, 1-n)$. The proof is left to the reader.

\begin{lemma} \label{lemma: alg Mukai classes}
Let $S$ be an algebraic $K3$ surface with $\Pic(S) = \IZ H$. An element $a \in H^*_{\textrm{alg}}(S, \IZ)$ has $a^2 = 2\rho$ and $(v,a) = \alpha$, with $\rho, \alpha \in \IZ$, if and only if it is of the form 
\[ a = \left( \frac{X + \alpha}{2(n-1)}, -YH , \frac{X - \alpha}{2} \right) \qquad \text{or} \qquad a = -\left( \frac{X - \alpha}{2(n-1)}, -YH , \frac{X + \alpha}{2} \right)\]
\noindent where $(X,Y)$ is a solution of Pell's equation $X^2 - 4t(n-1)Y^2 = \alpha^2 - 4\rho(n-1)$ such that $X \geq 0$ and $2(n-1) \mid (X + \alpha)$ or $2(n-1) \mid (X - \alpha)$ respectively.
\end{lemma}

Let $\post(\hsk)$ be the \emph{positive cone} of $\hsk$, i.e.\ the connected component of $\left\{ x \in \ns(\hsk) \otimes_\IZ \IR \mid (x,x) > 0 \right\}$ containing K\"ahler classes. The \emph{movable cone} of $\hsk$ is the (open) cone $\movb(\hsk) \subset \ns(\hsk) \otimes_\IZ \IR$ generated by the classes of divisors $D$ such that the linear system $\left| D \right|$ has no divisorial base locus. In particular, $\movb(\hsk) \subset \post(\hsk)$ and $\movb(\hsk)$ contains the ample cone $\mathcal{A}_{\hsk}$. The structures of these cones were studied extensively by Bayer and Macr{\`i}: we recall their main results. 

\begin{theorem}\cite[Theorem 12.3]{bayer_macri_mmp} \label{thm_bm_movb}
The movable cone $\movb(\hsk)$ is one of the open chambers of the decomposition of $\overline{\post(\hsk)}$ whose walls are the linear subspaces $\theta(v^\perp \cap a^\perp)$ for $a \in H^*_{\textrm{alg}}(S, \IZ)$ and
\begin{enumerate}
\item[\textit{(i)}] $a^2 = -2$, $(v,a) = 0$, or
\item[\textit{(ii)}] $a^2 = 0$, $(v,a) = 1$, or
\item[\textit{(iii)}] $a^2 = 0$, $(v,a) = 2$.
\end{enumerate}
\end{theorem}

\begin{theorem}\cite[Theorem 12.1]{bayer_macri_mmp} \label{thm_bm_ample}
The ample cone $\mathcal{A}_{\hsk}$ is one of the open chambers of the decomposition of $\overline{\post(\hsk)}$ whose walls are the linear subspaces $\theta(v^\perp \cap a^\perp)$ for $a \in H^*_{\textrm{alg}}(S, \IZ)$ with $a^2 \geq -2$ and $0 \leq (v, a) \leq n-1$.
\end{theorem}

The walls of Theorem \ref{thm_bm_ample} contained in $\movb(\hsk)$ provide a decomposition of the movable cone into chambers, which correspond bijectively to smooth birational models of $\hsk$ with trivial canonical bundle, as studied in a more general setting in \cite{hass_tschink_moving} and \cite{bayer_macri_mmp}.

\begin{rem} \label{rem: flops}
The additional walls that need to be considered in the decomposition of Theorem \ref{thm_bm_ample}, with respect to the ones already appearing in Theorem \ref{thm_bm_movb}, are \emph{flopping walls}. They are walls of the form $\theta(v^\perp \cap a^\perp)$, where we can restrict to consider $a$ algebraic such that:
\begin{enumerate}
\item  $a^2 = -2$ and $1 \leq (v,a) \leq n-1$, or
\item $a^2 = 0$ and $3 \leq (v,a) \leq n-1$, or
\item $2 \leq a^2 < \frac{n-1}{2}$ and $2a^2 + 1 \leq (v,a) \leq n-1$.
\end{enumerate}
These bounds for $a^2$ and $(v,a)$ are a consequence of the fact that the rank two sublattice $\langle a,v \rangle \subset H^*(S, \IZ)$ needs to be hyperbolic, in order to define a subspace $\theta(v^\perp \cap a^\perp)$ which intersects the positive cone (see \cite[Theorem 5.7]{bayer_macri_mmp}). Notice in particular that $\mathcal{A}_{\hsk} = \movb(\hsk)$ if and only if there are no flopping walls inside $\movb(\hsk)$.
\end{rem}

\begin{rem} \label{rem: cones1}
If $\sigma \in \bir(\hsk)$ is a birational automorphism of $\hsk$, then $\sigma^*$ leaves the closure of $\movb(\hsk)$ globally invariant and the nef cone $\nef(\hsk)$ (which is the closure of the ample cone in $\ns(\hsk) \otimes_{\IZ} \IR$) is mapped by $\sigma^*$ to one of the closed chambers of the decomposition of $\overline{\movb(\hsk)}$ determined by the walls of Theorem \ref{thm_bm_ample}. If moreover $\sigma$ is biregular (i.e.\ $\sigma \in \aut(X)$), then $\sigma^*(\nef(\hsk)) = \nef(\hsk)$.
\end{rem}

In the case of Hilbert schemes $\hsk$ with $\rk(\ns(\hsk)) = 2$, both the ample cone and the movable cone are delimited by two extremal rays.

\begin{theorem}\cite[Proposition 13.1]{bayer_macri_mmp} \label{thm_bm_mvb hsk}
Let $S$ be an algebraic $K3$ surface with $\Pic(S) = \IZ H$, $H^2 = 2t$ and $n \geq 2$.
\begin{enumerate}
\item If $t(n-1) = c^2$, for $c \in \mathbb{N}$, then $\overline{\movb(\hsk)} = \langle h, (n-1)h - c\delta \rangle$.
\item If $t(n-1)$ is not a square and the equation $(n-1)X^2 - tY^2 = 1$ has integer solutions, let $(z,w)$ be the positive solution with minimal $X$; then \mbox{$\overline{\movb(\hsk)} = \langle h, (n-1)zh - tw\delta \rangle$}.
\item If $t(n-1)$ is not a square and $(n-1)X^2 - tY^2 = 1$ has no integer solutions, then $\overline{\movb(\hsk)} = \langle h, zh - tw\delta \rangle$ where $(z,w)$ is the minimal solution of Pell's equation $X^2 - t(n-1)Y^2 = 1$ with $X \equiv \pm 1$ (mod $n-1$). 
\end{enumerate}
\end{theorem}

In the hypothesis of the theorem, the extremal ray of $\overline{\movb(\hsk)}$ generated by $h$ is also one of the two walls delimiting $\nef(\hsk)$. If $\nef(\hsk) \neq \overline{\movb(\hsk)}$, by Remark \ref{rem: flops} the second wall of the nef cone coincides with the flopping wall with minimal slope inside $\movb(\hsk)$. 

\section{Ample classes and isometries of $\ns ( S^{\left[ n\right]})$} \label{sec: isometries}

In this section we determine the structure of the group of isometries $O(\ns(\hsk))$. We adopt a similar approach to the one used by the authors of \cite{bcnws} for the case $n=2$, in order to generalize their results.

\begin{defi}
Let $S$ be a smooth complex surface and $k$ a non-negative integer. A line bundle $L \in \Pic(S)$ is \emph{$k$-very ample} if the restriction $H^0(S,L) \ra H^0(S, L \otimes \mathcal{O}_Z)$ is surjective for any zero-dimensional subscheme $(Z, \mathcal{O}_Z)$ of length $h^0(\mathcal{O}_Z) \leq k+1$.
\end{defi}

We recall the following geometric interpretation of $k$-very ampleness. For any subscheme $(Z, \mathcal{O}_Z)$ we have the inclusion $H^0(S, L \otimes I_Z) \hookrightarrow H^0(S,L)$, which allows us to define a rational map $\gamma: S^{\left[k\right]} \dashrightarrow \grass(k, H^0(S,L))$. Then, $\gamma$ is an embedding if and only if $L$ is $k$-very ample (\cite[Main Theorem]{catanese_gott}).

\begin{prop} \cite[Theorem 1.1]{knutsen} \label{knut very-ample}
Let $S$ be a $K3$ surface, $L$ a big and nef line bundle on $S$ and $k \geq 0$ an integer. Then $L$ is $k$-very ample if and only if $L^2 \geq 4k$ and there exist no effective divisors $D$ such that:
\begin{enumerate}
\item $2D^2 \stackrel{(i)}{\leq} L \cdot D \leq D^2 + k + 1 \stackrel{(ii)}{\leq} 2k+2$;
\item $(i)$ is an equality if and only if $L \sim 2D$ and $L^2 \leq 4k+4$;
\item $(ii)$ is an equality if and only if $L \sim 2D$ and $L^2 = 4k+4$.
\end{enumerate}
\end{prop}

\begin{prop} \label{ample_classes}
Let $S$ be an algebraic $K3$ surface with $\Pic(S)=\IZ H$, $H^2 = 2t$, $t \geq 1$. The class $ah - \delta \in \ns(\hsk)$ is ample if $a \geq n+2$. In particular, inside $\ns(\hsk) \otimes_\IZ \IR$, the ample cone $\mathcal{A}_{\hsk}$ is contained in $\left\{ xh - y\delta \mid x >0, y>0 \right\}$.
\end{prop}

\begin{proof}
If $L = aH \in \Pic(S)$ is $n$-very ample, then the element $ah - \delta \in \ns(\hsk)$ is ample, because it is the first Chern class of the line bundle $\gamma^*(\mathcal{O}_{\mathbb{G}}(1))$, where we consider the embedding $\gamma: \hsk \ra \mathbb{G} \coloneqq \grass(n, H^0(S,L))$ obtained from $L$ (see \cite[Construction 2]{betram_coskun}). We now use Knutsen's characterization of $n$-very ample line bundles (Proposition \ref{knut very-ample}). We have $L^2 \geq 4n$ if and only if $ta^2 \geq 2n$. If $D = dH$ is effective ($d > 0$), we can reformulate condition $(1)$ of Proposition \ref{knut very-ample} as
\[ 4td^2 \stackrel{(i)}{\leq} 2tad \leq 2td^2+n+1 \stackrel{(ii)}{\leq} 2n+2.\]
It is easy to see that, if $a \geq n+2$, there are no values $d$ satisfying all these inequalities, while instead the condition $L^2 \geq 4n$ holds.

Thus, for all $a \geq n+2$ the line bundle $L = aH$ is $n$-very ample and $ah - \delta \in \mathcal{A}_{\hsk}$. Moreover, the two classes $h,-\delta \in \ns(\hsk)$ are not ample, therefore we conclude $\mathcal{A}_{\hsk} \subseteq \left\{ xh - y\delta \mid x >0, y>0 \right\}$, because $\mathcal{A}_{\hsk}$ is a convex cone. 
\end{proof} 

We are now interested in describing the isometries of the N\'eron-Severi lattice of $\hsk$.
With respect to the basis $\left\{ h, -\delta \right\}$, the bilinear form on $\ns(\hsk)$ satisfies $(h,h) = 2t, (-\delta, -\delta) = -2(n-1), (h,-\delta) = 0$.  Thus, in coordinates, an isometry $\phi \in O(\ns(\hsk))$ is represented by a matrix of the form

\[ M = \begin{pmatrix}
A & B \\
C & D
\end{pmatrix} \]

\noindent with $A, B, C, D$ integers such that:

\begin{enumerate}
\item $\det M = \pm 1$, i.e.\ $AD - BC = \pm 1$;
\item $(Ah - C\delta, Ah - C\delta) = 2t$ , i.e.\ $(n-1)C^2 = t(A^2 - 1)$; 
\item $(Bh - D\delta, Bh - D\delta) = -2(n-1)$, i.e.\ $\left( n-1\right)(D^2 -1)= tB^2$;
\item $(Ah - C \delta, Bh - D\delta) = 0$, i.e.\ $\left( n-1\right)CD = tAB$.
\end{enumerate}

From this list of conditions, we find two alternative structures for $M$:
\[\begin{pmatrix}
A & B \\
C & A
\end{pmatrix} \; \text{or} \; \begin{pmatrix}
A & B \\
-C & -A
\end{pmatrix}, \quad \text{with } (n-1)C = tB \; \text{and } (n-1)A^2 - tB^2 = n-1.\]

Matrices of the first form have determinant $+1$, while those of the second form have determinant $-1$; moreover, for $M = \begin{pmatrix}
A & B\\
-C & -A
\end{pmatrix}$ we have $M^2 = I_2$ (from the relations between $A,B,C$), meaning that the isometry described by $M$ is an involution.

Notice that we can write $O( \ns ( S^{\left[ n\right]})) \cong N \rtimes \langle s \rangle$, where $N$ is the normal subgroup of isometries of the first form:
\[N := \left\{\begin{pmatrix}
A & B \\
C & A
\end{pmatrix} \mid A,B,C \in \mathbb{Z}, (n-1)C = tB, (n-1)A^2 - tB^2 = n-1 \right\} \]
\noindent while $s$ is the order two matrix $s := \begin{pmatrix}
1 & 0 \\
0 & -1
\end{pmatrix} \in O( \ns ( S^{\left[ n\right]}))$.

An automorphism $f \in \aut(\hsk)$ induces, by pull-back, an isometry $f^*$ of the lattice $H^2(\hsk, \IZ)$, such that $f^*\vert_{\ns(\hsk)} \in O(\ns(\hsk))$. As proved for the case $n=2$ in \cite[Proposition 4.3]{bcnws}, there is a link between such isometries and the ample cone of $\hsk$.

\begin{prop} \label{unique_isometry}
Let $S$ be an algebraic $K3$ surface with $\Pic(S) = \IZ H$, $H^2 = 2t$, $t \geq 1$. The isometry of $\ns(\hsk)$ induced by $f \in \aut(\hsk)$ is either the identity or the involution described, with respect to the basis $\left\{ h, -\delta \right\}$, by the matrix
\begin{equation*}
\begin{pmatrix}
A & B \\
-C & -A
\end{pmatrix}, \quad \text{with } (n-1)C = tB, \; (n-1)A^2 - tB^2 = n-1, \; A > 0, \; B<0
\end{equation*}
\noindent and with $A,B$ defining the ample cone of $\hsk$:
\[ \mathcal{A}_{\hsk} = \left\{ xh -y \delta \mid y>0, (n-1)Ay < -tBx\right\}.\] 
\end{prop}

\begin{proof}
The proof generalizes the one of \cite[Proposition 4.3]{bcnws}. The isometry $\phi$ induced by $f$ on the N\'eron-Severi group of $\hsk$ can only be of the two forms we described previously in this section. We look at them separately.
\begin{itemize}
\item $\phi = \begin{pmatrix}
A & B \\
C & A
\end{pmatrix}$ with $(n-1)C = tB$ and $(n-1)A^2 - tB^2 = n-1$.\\ Assume $\phi \neq \pm \id$, i.e.\ $B \neq 0$. Notice that $\phi$ is the restriction to $\ns(\hsk)$ of the effective Hodge isometry $f^* \in O(H^2(\hsk, \IZ))$, therefore it needs to map ample classes to ample classes. Then, by applying Proposition \ref{ample_classes}, $\phi(a,1) = (aA+B,aC+A) \in \mathcal{A}_{\hsk}$ for any $a \geq n+2$. Moreover, we have $\mathcal{A}_{\hsk} \subseteq \left\{ xh - y\delta \mid x >0, y>0 \right\}$, again by Proposition \ref{ample_classes}, so we deduce $A>0$, $C>0$ (therefore also $B>0$). Instead, since $h$ is not an ample class, $\phi(1,0) = (A,C) \notin \mathcal{A}_{\hsk}$, which means that (in the plane $\ns(\hsk) \otimes_\IZ \IR$):
\[ \mathcal{A}_{\hsk} \subseteq \left\{ xh - y\delta \mid y>0, Ay < \frac{tB}{n-1}x \right\}. \]  
Now, the class $\phi(n+2,1) = \left( (n+2)A + B, (n+2)C + A \right)$ needs to be in the ample cone, but it does not satisfy the inequality $Ay < \frac{tB}{n-1}x$ (because of the property $(n-1)A^2 - tB^2 = n-1 >0$), so we get a contradiction: $\phi$ cannot be of this form, unless $\phi = \id$ (we have to exclude $\phi = -\id$, because it does not preserve the ample cone).
\smallskip 

\item $\phi = \begin{pmatrix}
A & B \\
-C & -A
\end{pmatrix}$ with $(n-1)C = tB$ and $(n-1)A^2 - tB^2 = n-1$.\\ We proceed as in the previous case: the classes  $\phi(a,1) = (aA+B,-aC-A)$ are ample for any $a \geq n+2$, so $A>0$ and $C <0$ (thus also $B<0$; notice that we cannot have $C=0$, otherwise $\phi(a,1)$ would not be in the ample cone, for any positive $a$). In particular, all the rays through the classes $\phi(a,1)$, for $a \geq n+2$, are contained in the ample cone $\mathcal{A}_{\hsk}$; passing to the limit $a \ra +\infty$, the ray trough $(A, -C)$ must be in the closure of the cone, so
 \[ \mathcal{A}_{\hsk} \supseteq \mathcal{F} \coloneqq \left\{ xh - y\delta \mid y>0, Ay < -\frac{tB}{n-1}x \right\}. \]
 To conclude, we observe that $\phi(1,0) = (A,-C) \notin \mathcal{A}_{\hsk}$, because $h$ is not ample, so we also have $\mathcal{A}_{\hsk} \subseteq \mathcal{F}$. \qedhere
\end{itemize}
\end{proof}

\section{The automorphism group of $S^{\left[ n\right]}$} \label{section: autom group}

The aim of this section is to classify the possible group structures of $\aut(\hsk)$ and to determine some first numerical conditions for the existence of a non-trivial automorphism.  

For any $K3$ surface $S$ and $n \geq 2$, an element $\sigma \in \aut(S)$ induces an automorphism $\sigma^{[n]}$ of the Hilbert scheme $S^{[n]}$, which maps a zero-dimensional subscheme $Z \subset S$ of length $n$ to its schematic image $\sigma(Z)$. Such an automorphism $\sigma^{[n]}$ is called \emph{natural}.

We recall that, if $S$ is a projective $K3$ surface with $\Pic(S)=\IZ H$, $H^2 = 2t$, the automorphism group $\aut(S)$ is trivial if $t \geq 2$. Instead, if $t=1$ there exists a double covering $S \ra \mathbb{P}^2$ which is ramified over a smooth curve of degree six; in particular, $\aut(S) = \left\{ \id, \iota \right\}$, where $\iota$ is the (non-symplectic) covering involution (see \cite[\S 5]{saint-donat} and \cite[Corollary 15.2.12]{huybrechts}). Proposition \ref{unique_isometry} allows us to provide a first result on the structure of the automorphism group $\aut(\hsk)$.

\begin{prop} \label{aut for t>=2}
Let $S$ be an algebraic $K3$ surface with $\Pic(S)=\IZ H$, $H^2 = 2t$. If $t \geq 2$, the automorphism group $\aut(\hsk)$ is either trivial or isomorphic to $\IZ/2 \IZ$, generated by a non-natural, non-symplectic involution.
\end{prop}

\begin{proof}
The map $\aut(\hsk) \ra O(H^2(\hsk, \IZ))$, sending an automorphism of $\hsk$ to its action on the second cohomology lattice, is injective by \cite[Proposition 10]{beauville_eng}; instead, if we only consider $\Psi: \aut(\hsk) \ra O(\ns(\hsk))$, $f \mapsto f^*\vert_{\ns(\hsk)}$, its kernel is the set of natural automorphisms (generalization of \cite[Lemma 2.4]{bcnws}, using \cite[Theorem 1]{bs}). Under the hypothesis $t \geq 2$, the identity is the only automorphism of $S$ therefore $\aut(\hsk)$ is in one-to-one correspondence with the image of $\Psi$, which -- by Proposition \ref{unique_isometry} -- is either trivial or generated by an isometry of order two.

If $\aut(\hsk)$ contains an involution $f$, then it is non-natural and non-symplectic: in fact, if it were symplectic the co-invariant lattice $\left(H^2(\hsk, \IZ)^{f^*}\right)^\perp \subset H^2(\hsk, \IZ)$ would be contained in $\ns(\hsk)$ and it would be of rank eight (see \cite[Lemma 3.5, Corollary 5.1]{mongardi}), while here $\rk( \ns(\hsk)) = 2$. 
\end{proof}

\begin{rem}
If $t=1$, the map $\Psi$ is no longer injective: $\ker(\Psi) = \left\{\id, \iota^{[n]} \right\}$, where $\iota^{[n]}$ is the natural (non-symplectic) automorphism of order two induced by the covering involution $\iota \in \aut(S)$. Nevertheless, any automorphism $f \in \aut(\hsk) \setminus \ker(\Psi)$ is a non-natural, non-symplectic involution and the quotient $\aut(\hsk) /\ker(\Psi)$ is either trivial or isomorphic to $\IZ / 2 \IZ$, following the proof of Proposition \ref{aut for t>=2}.
\end{rem}

Let now $f \in \aut(\hsk)$ be a non-natural automorphism (that is, if $t \geq 2$, any non-trivial automorphism); then, as we showed in Proposition \ref{unique_isometry}, it induces an isometry of $\ns(\hsk)$ of the form
\begin{equation} \label{forma}
\phi = \begin{pmatrix}
A & B \\
-C & -A
\end{pmatrix}, \quad (n-1)C = tB, \; (n-1)A^2 - tB^2 = n-1 
\end{equation}
\noindent with coefficients $A >0,B<0,C<0$ uniquely determined by the ample cone $\mathcal{A}_{\hsk}$.

One can check that this isometry is the reflection of the N\'eron-Severi lattice in the line spanned by the class of coordinates $(-B,A-1)$. As it was done in \cite{bcnws} for the case $n=2$, we denote by $(b,a)$ the primitive generator of this line: in particular, $(b,a) \coloneqq \frac{1}{d}(-B, A-1)$, with $d = \gcd(-B,A-1)$. By computing explicitly the matrix of the reflection whose fixed line is $\langle (b,a) \rangle$, we find the relations
\begin{equation} \label{A_B_a_b}
 A = \frac{tb^2 + (n-1)a^2}{tb^2 - (n-1)a^2}, \qquad B = -\frac{2ab(n-1)}{tb^2 - (n-1)a^2}.
\end{equation}

Since $f$ is a non-symplectic involution, the invariant lattice $T_f \coloneqq H^2(\hsk, \IZ)^{f^*}$ is contained in $\ns(\hsk)$; therefore, $T_f$ is the lattice of rank one generated by $(b,a)$. Moreover, the transcendental lattice $\transc(\hsk) = \ns(\hsk)^\perp \subset H^2(\hsk, \IZ)$ is contained in $T_f^\perp$; in particular, $f^* \vert_{\transc(\hsk)} = -\id$, because the non-natural involution $f^*$ acts as $-\id$ on $T_f^\perp$ (for more details see for instance \cite[\S 5,6]{bnws_smith}).

\begin{lemma} \label{extending isometry}
Let $\phi$ be the isometry of $\ns(\hsk)$ of the form (\ref{forma}) induced by a non-natural automorphism $f \in \aut(\hsk)$; let $(b,a)$ be the primitive generator of the invariant lattice $T_f \subseteq \ns(\hsk)$. Then:
\begin{enumerate}
 \item[\textit{(i)}] $tb^2 - (n-1)a^2$ divides $b$;
 \item[\textit{(ii)}] there exists an even negative integer $\beta$ such that $B = (n-1)\beta$.
\end{enumerate}
\end{lemma}

\begin{proof}
The isometry $f^* \in O(H^2(\hsk, \IZ))$ has restrictions $f^*\vert_{\ns(\hsk)} = \phi$ and $f^*\vert_{\transc(\hsk)} = -\id_{\transc(\hsk)}$. The following isometry of lattices holds:
$$H^2(\hsk, \IZ) \cong L \coloneqq U^{\oplus 3} \oplus E_8(-1)^{\oplus 2} \oplus \langle -2(n-1) \rangle,$$
\noindent where $\langle -2(n-1) \rangle$ denotes the rank one lattice generated by an element of square $-2(n-1)$. Following the proof of \cite[Lemma 5.2]{bcnws}, we embed $\ns(\hsk) = \IZ h \oplus \IZ(-\delta)$ inside $L$ mapping $h$ to $e_1+te_2$ (where $\left\{ e_1, e_2 \right\}$ is a basis for the first summand $U$) and $-\delta$ to the generator $g$ of the component $\langle -2(n-1) \rangle$. We take $\left\{ e_1+te_2, e_2, g \right\}$ as a basis for the lattice $U \oplus \langle -2(n-1) \rangle$, which contains $\ns(\hsk)$ as a sublattice of rank two. Notice that the class $w \coloneqq e_1 - te_2$ belongs to the transcendental lattice, therefore $f^*(w) = -w$. By writing $w = (e_1+te_2)-2te_2$ we compute:
\[ 2tf^*(e_2) = (A+1)(e_1+te_2) -2te_2 - Cg \]
\noindent which can also be written, using relations (\ref{A_B_a_b}), as
\[ f^*(e_2) = \frac{b^2}{tb^2 - a^2(n-1)}(e_1+te_2) - e_2 + \frac{ab}{tb^2 - a^2(n-1)}g.\]

Since the coefficients of this expression need to be integers, $tb^2 - (n-1)a^2$ divides $\gcd(b^2, ab) = b$. Also notice that, due to the condition $(n-1)A^2 - tB^2 = n-1$, we have: $d(tb^2 - a^2(n-1)) = 2a(n-1)$. Therefore, since $\frac{ab}{tb^2 - a^2(n-1)}$ is integer, by multiplying both numerator and denominator by $d$ we deduce that $2(n-1)$ divides $B=-db < 0$.
\end{proof}

Using Lemma \ref{extending isometry}, we can determine the group $\aut(\hsk)$ in some cases where we have a simple description of the nef cone of $\hsk$.

\begin{prop} \label{caso speciale}
Let $S$ be an algebraic $K3$ surface with $\Pic(S)=\IZ H$, $H^2 = 2t$, $t \geq 1$. If $n \geq \frac{t+3}{2}$, all automorphisms of $\hsk$ are natural.
\end{prop}

\begin{proof}
Let $f \in \aut(\hsk)$ be a non-natural automorphism: its action on $\ns(\hsk)$ is non-trivial, therefore it is an isometry of the form (\ref{forma}), by Proposition \ref{unique_isometry}. In particular, the nef cone $\nef(\hsk)$ is generated over $\IR^+$ by $h$ and $(n-1)Ah + tB\delta$. However, if $n \geq \frac{t+3}{2}$ we also have the description of the nef cone given in \cite[Proposition 10.3]{bayer_macri_vecchio}: it is generated by $h$ and $h-\frac{2t}{t+n}\delta$. Thus, we need the two classes $(n-1)Ah + tB\delta$ and $h-\frac{2t}{t+n}\delta$ to be proportional, i.e.\
\[ \frac{(n-1)A}{t+n} = -\frac{tB}{2t}.\]

Using relations (\ref{A_B_a_b}) this becomes $tb^2 + (n-1)a^2 = (t+n)ab$, therefore
\begin{equation*}
tb^2 - (n-1)a^2 = (t+n)ab - 2(n-1)a^2=a((t+n)b-2(n-1)a).
\end{equation*}

As we proved in Lemma \ref{extending isometry}, $tb^2 - (n-1)a^2$ divides $b$; but, from the last expression for $tb^2 - (n-1)a^2$, this implies $a \mid b$, so necessarily $a=1$, since $\gcd(a,b)=1$ by definition (also, $a$ and $b$ are both positive).

Therefore, $b$ is an integer solution of the equation: 
\begin{equation*}
tb^2 - (t+n)b + n-1 = 0.
\end{equation*}

Assuming $n \geq 2$, a simple computation shows that this equation admits an integer solution only if $n$ is odd and $n+1 = 2t$: in this case, $b=2$. However, if so we have $tb^2 - (n-1)a^2 = 2t+2$, which does not divide $b=2$ for any $t \geq 1$, thus we get a contradiction.
\end{proof}

Notice that the condition $n \geq \frac{t+3}{2}$ is satisfied by all values $n \geq 2$ when $t=1$, therefore the case $t=1$ is now completely resolved. If instead $t \geq 2$, we already know that there are no non-trivial natural automorphisms on $\hsk$.

\begin{cor} \label{cor: caso speciale}
Let $S$ be an algebraic $K3$ surface with $\Pic(S)=\IZ H$, $H^2 = 2t$.
\begin{itemize}
\item If $t=1$, then $\aut(\hsk) = \left\{ \id, \iota^{[n]} \right\}$, where $\iota^{[n]}$  is the natural involution induced by the covering involution $\iota \in \aut(S)$.
\item If $2 \leq t \leq 2n-3$, then $\aut(\hsk) = \left\{ \id \right\}$.
\end{itemize}
\end{cor}

From now on, we will assume $t \geq 2$. In Proposition \ref{aut for t>=2} and Lemma \ref{extending isometry} we proved that the isometry $\phi \in O(\ns(\hsk))$ induced by a non-trivial automorphism of $\hsk$ is of the form (\ref{forma}) with $B = (n-1)\beta$ and $\beta < 0$ even. We therefore rewrite the structure of the matrix of $\phi$ and the conditions on its coefficients in the following way:
\begin{equation} \label{forma_new}
\phi = \begin{pmatrix}
A & (n-1)\beta \\
-t\beta & -A
\end{pmatrix}, \quad \text{with } A^2 - t(n-1)\beta^2 = 1, \; A > 0, \; \beta < 0 \textrm{ even} 
\end{equation}

\noindent and with $\mathcal{A}_{\hsk} = \left\{ xh - y\delta \mid y>0, Ay < -t \beta x \right\}$. We will return to numerical conditions for the existence of an automorphism on $\hsk$ in \S \ref{sec: final numerical conditions}.

\section{Invariant polarizations} \label{sec: invariant pol}

In this section we study the properties of the generator of the invariant lattice of a non-natural involution on $\hsk$ and, vice versa, we show that the existence of an ample divisor in $\ns(\hsk)$ with such properties guarantees that $\aut(\hsk)$ is non-trivial.

Let $M$ be an even lattice; the dual $M^\vee \coloneqq \Hom_ {\IZ}(M, \IZ)$ admits the following description:
\[M^\vee = \left\{ u \in M \otimes \IQ : (u,v) \in \IZ \;\; \forall v \in M \right\} \]
\noindent where the brackets $(-,-)$ denote the $\IQ$-linear extension of the bilinear form of $M$. The finite quotient $A_M \coloneqq M^\vee / M$ is called the \emph{discriminant group} of $M$; for instance, the discriminant group of $L = U^{\oplus 3} \oplus E_8(-1)^{\oplus 2} \oplus \langle -2(n-1) \rangle$ is cyclic of order $2(n-1)$. The quadratic form on $M^\vee$ (coming from the one of $M$) descends to a finite quadratic form (with values in $\IQ/2\IZ$) on $A_M$. Any $\varphi \in O(M)$ induces -- in a canonical way -- an isometry $\varphi\vert_{A_M}$ of the discriminant group (see \cite[\S 1.2]{dolgachev} for more details). 

The \emph{divisibility} $\divi(m)$ of an element $m \in M$ is the positive generator of the ideal $\left\{(m,x) \mid x \in M \right\} \subset \IZ$.
For the remaining of the paper, when using the notation $\divi(x)$ we will always refer to the divisibility in the lattice $L \cong H^2(\hsk, \IZ)$, even if the element $x$ is stated to belong to a proper sublattice of $L$, such as $\ns(\hsk)$. 

\begin{prop} \label{prop: generatori T}
Let $S$ be an algebraic $K3$ surface with $\Pic(S)=\IZ H$, $H^2 = 2t$, $t \geq 2$ and $f \in \aut(\hsk)$ an involution. Let $\nu$ be the primitive generator of the rank one invariant lattice $T_f \subset H^2(\hsk, \IZ)$. Then one of the following holds:
\begin{itemize}
\item $f^*$ acts as $-\id$ on the discriminant group of $H^2(\hsk, \IZ)$ and $(\nu, \nu) = 2$;
\item $-1$ is a quadratic residue modulo $n-1$, $f^*$ acts as $\id$ on the discriminant group of $H^2(\hsk, \IZ)$, $(\nu, \nu) = 2(n-1)$ and $\divi(\nu) = n-1$.
\end{itemize}
\end{prop}

\begin{proof}
The generator $\nu$ of $T_f$ coincides with the ample class $bh - a\delta \in \ns(\hsk)$ defined in Section \ref{section: autom group}. We recall that $f^*\vert_{\ns(\hsk)}$ is the reflection fixing the line $\langle \nu \rangle$, while $f^*$ acts as $-\id$ on $\transc(\hsk)$, therefore $f^*$ can also be regarded as the opposite of the reflection of $H^2(\hsk, \IZ)$ defined by $\nu$, i.e.\
\begin{equation} \label{eqt: reflection}
f^* = -R_{\nu}: H^2(\hsk, \IZ) \rightarrow H^2(\hsk, \IZ), \quad m \mapsto 2\frac{(m,\nu)}{(\nu,\nu)}\nu-m.
\end{equation}
Moreover, $f^*$ is a monodromy operator (it is the isometry of $H^2(\hsk, \IZ)$ induced by an automorphism of $\hsk$; see \cite[Definition 1.1]{markman}). If we identify $H^2(\hsk, \IZ)$ with the lattice $L$ via the isometry presented in the proof of Lemma \ref{extending isometry}, by \cite[Lemma 9.2]{markman} we have 
\[f^* \in \widetilde{O}^+(L) \coloneqq \left\{ \varphi \in O^+(L) \mid \varphi\vert_{A_L} = \pm \id \right\}\]
\noindent where $O^+(L)$ is the subgroup of isometries of $L$ with real spinor norm $+1$ (i.e.\ the isometries which preserve the positive cone).

The vectors $l \in L$ such that $R_l\vert_{A_L} = \id$ or $(-R_l)\vert_{A_L} = \id$ were studied in \cite[Proposition 3.1]{ghs_kodaira}, \cite[Proposition 3.2]{ghs_kodaira} respectively: by applying these results, and observing that $\nu$ has square $(\nu, \nu) = 2(tb^2 - (n-1)a^2) > 0$, we can conclude the following:
\begin{enumerate}
 \item[\textit{(i)}] if $f^*\vert_{A_L} = -\id$, then $(\nu, \nu) = 2$;
 \item[\textit{(ii)}] if $f^*\vert_{A_L} = \id$, then $(\nu, \nu) = 2(n-1)$ and $\divi(\nu)=n-1$ or $\divi(\nu)=2(n-1)$.
\end{enumerate}

We deduce that the square of the generator of the invariant lattice $T_f$ can only be $2$ or $2(n-1)$. The existence of a primitive element $l \in L$ with given square and divisibility depends on whether some arithmetic conditions are satisfied: this is proved in \cite[Proposition 3.6]{ghs}, where an explicit description of the lattice $l^\perp$ is also provided. Using this result we find out that case $(ii)$ is admissible only if $-1$ is a quadratic residue modulo $n-1$ and the divisibility of $\nu$ is $n-1$. This concludes the proof.
\end{proof}

\begin{rem} \label{rem: A cong pm 1}
By Proposition \ref{prop: generatori T}, the square of the class $\nu = bh -a\delta$ can be determined by looking at the action of $f^*$ on $A_L$. A generator for $A_L \cong A_{\langle -2(n-1) \rangle}$ is given by the class, in the quotient $L^\vee/L$, of $-\frac{\delta}{2(n-1)} \in L \otimes \IQ$, because $-\delta$ is the generator of the summand $\langle -2(n-1) \rangle$ in $L$, via the embedding $\ns(\hsk) \hookrightarrow L$ presented in the proof of Lemma \ref{extending isometry}. Thus, since $f^*\vert_{\ns(\hsk)} = \phi$ has the form (\ref{forma_new}), we obtain:
$$ f^*\vert_{A_L} \left( -\frac{\delta}{2(n-1)} \right) = \frac{1}{2(n-1)}\left( (n-1)\beta h -A(-\delta) \right) \equiv -\frac{A}{2(n-1)}(-\delta) \; (\text{mod } L) $$
\noindent where we used the fact that $\beta$ is even, by Lemma \ref{extending isometry}. On the other hand, we know that $f^*$ acts as $\pm \id$ on $A_L$, being a monodromy operator, hence either $A+1$ or $A-1$ is divisible by $2(n-1)$. In particular, if $2(n-1)$ divides $A-1$, then $f^*\vert_{A_L} = - \id$ and $\nu$ has square $2$; instead, if $2(n-1)$ divides $A+1$ then $f^*\vert_{A_L} = + \id$ and $\nu$ has square $2(n-1)$. Notice that $2(n-1)$ can divide both $A-1$ and $A+1$ only for $n=2$: in this case, $-\id = + \id$ on $A_L \cong \IZ/2\IZ$ and $\nu$ has square $2 = 2(n-1)$, as already proved in \cite[Theorem 1.1]{bcnws}.
\end{rem}

By Proposition \ref{prop: generatori T}, if $t \geq 2$ and $\aut(\hsk) \neq \left\{ \id \right\}$ there exists an ample class $\nu \in \ns(\hsk)$ with $(\nu,\nu) = 2$ or $(\nu,\nu) = 2(n-1)$, $\divi(\nu) = n-1$. We now show that the converse holds for any manifold of $\hskn$-type.

\begin{prop} \label{prop: existence of involutions}
Let $X$ be an hyperk\"ahler manifold deformation equivalent to the Hilbert scheme of $n$ points on a $K3$ surface, $n \geq 2$. Then $X$ admits a non-symplectic involution if there exists a primitive ample class $\nu \in \ns(X)$ with either
\begin{itemize}
\item $(\nu,\nu) = 2$, or
\item $(\nu,\nu) = 2(n-1)$ and $\divi(\nu) = n-1$.
\end{itemize}
\end{prop}

\begin{proof}
Let $\gamma \coloneqq -R_\nu \in O(H^2(X,\IQ))$ be the opposite of the reflection defined by $\nu$, as in (\ref{eqt: reflection}). If $\nu$ is an element as in the statement, $\gamma$ defines an integral isometry $\gamma \in O(H^2(X, \IZ))$. Moreover, $\gamma$ induces $\pm\id$ on the discriminant group of $H^2(X, \IZ)$ by \cite[Proposition 3.1, Proposition 3.2]{ghs_kodaira} and it belongs to $O^+(H^2(X, \IZ))$ by \cite[\S 9]{markman}, because $\gamma = -R_\nu$ with $(\nu,\nu) > 0$. This implies that $\gamma$ is a monodromy operator by \cite[Lemma 9.2]{markman}. Let $\omega_X$ be the everywhere non-degenerate closed two-form on $X$ which generates $H^{2,0}(X)$; then, after extending $\gamma$ to $H^2(X, \IC)$ by $\IC$-linearity, $\gamma(\omega_X)=-\omega_X$, since $\omega_X$ belongs to $\ns(X)^\perp$, thus $\gamma$ is an isomorphism of integral Hodge structures. Moreover, $\gamma$ preserves the K\"ahler cone of $X$, because it fixes the ample class $\nu$: by the Hodge theoretic global Torelli theorem \cite[Theorem 1.3]{markman}, we can conclude that there exists an automorphism $f \in \aut(X)$ whose action on $H^2(X,\IZ)$ is $\gamma$. In particular, $f$ is a non-symplectic involution, since the map $\aut(X) \ra O(H^2(X, \IZ))$ is injective and $\gamma(\omega_X)=-\omega_X$.
\end{proof}

\begin{rem} \label{remark: invol is non-nat}
If $S$ is a $K3$ surface of Picard number one and $X=\hsk$, the non-symplectic involution $f \in \aut(X)$ constructed in the proof of Proposition \ref{prop: existence of involutions} is non-natural, since its action on $\ns(\hsk)$ is non-trivial. As a consequence, using also Proposition \ref{prop: generatori T}, the existence of a primitive ample class with square two, or with square $2(n-1)$ and divisibility $n-1$, is equivalent to the existence of a non-symplectic, non-natural involution (as formulated in Theorem \ref{thm_intro: divisorial conditions}).
\end{rem}

\section{Numerical conditions} \label{sec: final numerical conditions}

In this last section we apply the divisorial results of \S \ref{sec: invariant pol} to provide a completely numerical characterization for the existence of a non-trivial automorphism on $\hsk$.

\begin{prop} \label{prop: numerical conditions and cones}
Let $S$ be an algebraic $K3$ surface with $\Pic(S)=\IZ H$, $H^2 = 2t$, $t \geq 2$. If $\aut(\hsk) \neq \left\{ \id \right\}$, then $t(n-1)$ is not a square and, if $n \neq 2$, the equation $(n-1)X^2 - tY^2 = 1$ has no integer solutions. The minimal solution $(z,w)$ of Pell's equation $X^2 - t(n-1)Y^2 = 1$ with $X \equiv \pm 1$ (mod $n-1$) has $w$ even and $z \equiv \pm 1$ (mod $2(n-1)$). Moreover, 
\[ \nef(\hsk) = \overline{\movb(\hsk)} = \langle h, zh-tw\delta \rangle.\]
The automorphism group of $\hsk$ is generated by a non-natural, non-symplectic involution whose action on $\ns(\hsk) = \IZ h \oplus \IZ (-\delta)$ is given by the matrix
\begin{equation*} \label{forma_ultimate}
\begin{pmatrix}
z & -(n-1)w \\
tw & -z
\end{pmatrix}.
\end{equation*}
\end{prop}

\begin{proof}
If the automorphism group is non-trivial, by Proposition \ref{aut for t>=2} it is generated by a non-natural, non-symplectic involution whose action on $\ns(\hsk)$ is of the form (\ref{forma_new}). In particular, the pair $(A, \beta)$ is a solution of Pell's equation  $X^2 - t(n-1)Y^2 = 1$ with $\beta \neq 0$: as a consequence, $t(n-1)$ cannot be a square (as we remarked in \S \ref{sec: pell}). Moreover, since the action of the involution on $\ns(\hsk)$ is non-trivial, it needs to exchange the two extremal rays of $\movb(\hsk)$, which therefore need to be of the same type with respect to the classification of walls of Theorem \ref{thm_bm_movb}. The extremal ray generated by $h$ corresponds to a wall $\theta(v^\perp \cap a^\perp)$ with $a \in H^*_{\textrm{alg}}(S, \IZ)$ isotropic such that $(v,a) = 1$: in \cite[Theorem 5.7]{bayer_macri_mmp} it is referred to as a wall of Hilbert--Chow type, since the corresponding divisorial contraction is the Hilbert--Chow morphism $\hsk \ra S^{(n)}$. Thus, the second wall of $\movb(\hsk)$ needs to be of this type too. This happens if and only if $\movb(\hsk)$ is as in case (3) of Theorem \ref{thm_bm_mvb hsk}, where in particular -- by Lemma \ref{lemma: alg Mukai classes} -- we need to ask that the minimal solution $(z,w)$ of $X^2 - t(n-1)Y^2 = 1$ with $X \equiv \pm 1$ (mod $n-1$) is such that $w$ is even and $z \equiv \pm 1$ (mod $2(n-1)$). If so, $\movb(\hsk)$ is the interior of the cone spanned by $h$ and $zh-tw\delta$ (Theorem \ref{thm_bm_mvb hsk}).

We know that $\mathcal{A}_{\hsk} = \left\{ xh - y\delta \mid y>0, Ay < -t \beta x \right\}$, with $\beta < 0$ even and $A > 0$, therefore the two extremal rays of the nef cone are generated by $h$ and $Ah - t(-\beta)\delta$. Moreover, either $A+1$ or $A-1$ is divisible by $2(n-1)$, by Remark \ref{rem: A cong pm 1}, meaning $A \equiv \pm 1$ (mod $(n-1)$). Since $\nef(\hsk) \subseteq \overline{\movb(\hsk)}$, the minimality of the slope $\frac{w}{z}$ implies $\nef(\hsk) = \overline{\movb(\hsk)}$, thus $A = z, \beta = -w$.
\end{proof}

\begin{rem}
If $n=2$, the equality $\nef(\hsk) = \overline{\movb(\hsk)}$ when the automorphism group is non-trivial was already proved in \cite[Theorem 1.1]{bcnws}, using \cite[Lemma 13.3]{bayer_macri_mmp}. In general, for any $n$, the equality also follows from the fact that the non-trivial action on $\ns(\hsk)$ of a biregular involution needs to exchange the two extremal rays of both $\movb(\hsk)$ and $\nef(\hsk)$ (see Remark \ref{rem: cones1}). Since the two cones share one of the extremal rays (the wall spanned by $h$), they share the other one too. 
\end{rem}

In order to convert the divisorial results of \S \ref{sec: invariant pol} into purely numerical conditions, we will need the following lemma. 

\begin{lemma} \label{lemma: 2(n-1)}
Let $S$ be an algebraic $K3$ surface with $\Pic(S)=\IZ H$, $H^2 = 2t$, $t \geq 1$. A primitive element $\nu \in \ns(\hsk)$ has $(\nu, \nu) = 2(n-1)$ and $\divi(\nu) = n-1$ if and only if it is of the form $\nu = (n-1)Yh - X\delta$, where $(X,Y)$ is a solution of Pell's equation $X^2 - t(n-1)Y^2 = -1$.
\end{lemma}

\begin{proof}
Assume that there exists $\nu = bh - a\delta \in \ns(\hsk)$ as in the statement: since $(\nu, \nu) = -2((n-1)a^2 - tb^2) = 2(n-1)$, we deduce that $(a,b)$ is an integer solution of $(n-1)X^2 - tY^2 = -(n-1)$. The canonical embedding of $\ns(\hsk)$ inside $H^2(\hsk,\IZ) \cong U^{\oplus 3} \oplus E_8(-1)^{\oplus 2} \oplus \langle -2(n-1)\rangle$ (see the proof of Lemma \ref{extending isometry}) maps $\nu$ to $b(e_1 + te_2) + ag$, where $\left\{e_1,e_2 \right\}$ is a basis for one of the summands $U$ and $g$ generates $\langle -2(n-1) \rangle$. Notice that, with respect to the bilinear form on $H^2(\hsk, \IZ)$, we have $(\nu, e_2) = b$, therefore $b$ is a multiple of $\divi(\nu) = n-1$. We conclude that $(a, \frac{b}{n-1})$ is an integer solution of $X^2 - t(n-1)Y^2 = -1$.

Conversely, assume that $X^2 - t(n-1)Y^2 = -1$ admits integer solutions and let $(X,Y)$ be one of them. We set $\nu \coloneqq (n-1)Yh - X\delta \in \ns(\hsk)$: it is a primitive class of square $2(n-1)$. Using the usual embedding $\ns(\hsk) \hookrightarrow H^2(\hsk, \IZ)$, which maps $\nu$ to $(n-1)Y(e_1+te_2) + Xg$, we can easily compute the ideal of pairings
\[ \left\{ (\nu, m) \mid m \in H^2(\hsk, \IZ) \right\} = \left\{ (n-1)(Yp-2Xq) \mid p,q \in \IZ\right\} \subset \IZ.\]
From this description it is clear that $(n-1) \mid \divi(\nu)$; moreover, since $\divi(\nu) \mid (\nu, \nu)$, we conclude that $\divi(\nu) = n-1$ if $Y$ is odd, while $\divi(\nu) = 2(n-1)$ if $Y$ is even. However, there are no solutions $(X,Y)$ to $X^2 - t(n-1)Y^2 = -1$ with $Y$ even, since $-1$ is not a quadratic residue modulo $4$, thus $\divi(\nu) = n-1$.
\end{proof}

We can now state and prove our main result.

\begin{theorem} \label{thm: numerical characterization}
Let $S$ be an algebraic $K3$ surface with $\Pic(S)=\IZ H$, $H^2 = 2t$, $t \geq 2$ and $n \geq 2$. Let $(z,w)$ be the minimal solution of Pell's equation $X^2 - t(n-1)Y^2 = 1$ with $X \equiv \pm 1$ (mod $n-1$). The Hilbert scheme $\hsk$ admits a (non-symplectic, non-natural) involution if and only if 
\begin{enumerate}
\item[\textit{(i)}] $t(n-1)$ is not a square;
\item[\textit{(ii)}] if $n \neq 2$, the equation $(n-1)X^2 - tY^2 = 1$ has no integer solutions;
\item[\textit{(iii)}] for all integers $\rho, \alpha$ as follows:
\begin{enumerate}
\item $\rho = -1$ and $1 \leq \alpha \leq n-1$, or
\item $\rho = 0$ and $3 \leq \alpha \leq n-1$, or
\item $1 \leq \rho < \frac{n-1}{4}$ and $\max\left\{4\rho+1,\ceil[\bigg]{2\sqrt{\rho(n-1)}}  \right\} \leq \alpha \leq n-1$
\end{enumerate}
the minimal solution $(X,Y)$ of Pell's equation 
\[X^2 - 4t(n-1)Y^2 = \alpha^2 - 4\rho(n-1)\]
with $X \equiv \pm \alpha$ (mod $2(n-1)$), if it exists, is such that $\frac{Y}{X} \geq \frac{w}{2z}$;
\item[\textit{(iv)}] there exist integer solutions either for the equation $(n-1)X^2 - tY^2 = -1$ or for the equation $X^2 - t(n-1)Y^2 = -1$.
\end{enumerate}
\end{theorem}

\begin{proof}
By Proposition \ref{prop: numerical conditions and cones}, if there exists a non-trivial automorphism on $\hsk$ then $t(n-1)$ is not a square and -- if $n \geq 3$ -- the equation $(n-1)X^2 - tY^2 = 1$ has no integer solutions; moreover, $\nef(\hsk) = \overline{\movb(\hsk)} = \langle h, zh-tw\delta \rangle$. The equality $\nef(\hsk) = \overline{\movb(\hsk)}$ implies that there are no flopping walls inside $\movb(\hsk)$. In \S \ref{sec: intro cones} we recalled the description of the elements $a \in H^*_{\textrm{alg}}(S, \IZ)$ such that $\theta(v^\perp \cap a^\perp)$ is a flopping wall: using Lemma \ref{lemma: alg Mukai classes}, the existence of a similar algebraic Mukai vector $a$ corresponds to the existence of integer solutions $(X, Y)$ with $X \equiv \pm \alpha$ (mod $2(n-1)$) to one of Pell's equations $X^2 - 4t(n-1)Y^2 = \alpha^2 - 4\rho(n-1)$, where the possible values of $a^2 = 2\rho$ and $(v,a) = \alpha$ were listed in Remark \ref{rem: flops}. In particular, a solution $(X,Y)$ for one of these equations yields the wall generated by $Xh - 2tY\delta$: this wall lies in the movable cone if and only if $X, Y > 0$ and $\frac{Y}{X} < \frac{w}{2z}$. By Remark \ref{rem: pell slope}, we can restrict to consider the equations where the RHS term $\alpha^2 - 4\rho(n-1)$ is positive. In fact, since $w$ is even by Proposition \ref{prop: numerical conditions and cones}, the pair $(z, \frac{w}{2})$ is the minimal solution of $X^2 - 4t(n-1)Y^2 = 1$, therefore
\begin{equation} \label{eq: slope w/2z}
\frac{w}{2z} = \sqrt{\frac{1}{4t(n-1)} - \frac{1}{4t(n-1)z^2}}
\end{equation}
\noindent which is strictly smaller than $\frac{Y}{X} = \sqrt{\frac{1}{4t(n-1)} - \frac{m}{4t(n-1)X^2}}$ for all solutions $(X,Y)$ of Pell's equation $X^2 - 4t(n-1)Y^2 = m$, if $m \leq 0$. 
We also notice that, for each equation with $\alpha^2 - 4\rho(n-1) > 0$, it is sufficient to check whether the wall defined by the minimal solution (with $X \equiv \pm \alpha$ modulo $2(n-1)$) lies inside the movable cone, since the other walls corresponding to positive solutions of the same equation will all have greater slopes (Remark \ref{rem: pell slope}).

Finally, if $\hsk$ admits an involution then, by Proposition \ref{prop: generatori T}, there exists a primitive element $\nu  \in \ns(\hsk)$ with either $(\nu, \nu) = 2$ or $(\nu, \nu) = 2(n-1)$ and $\divi(\nu) = n-1$. If we write $\nu = bh - a\delta$, we have $(\nu, \nu) = -2((n-1)a^2 - tb^2)$.
\begin{itemize}
\item If $(\nu, \nu) = 2$, then $(a,b)$ is an integer solution of $(n-1)X^2 - tY^2 = -1$.
\item If $(\nu, \nu) = 2(n-1)$ and $\divi(\nu) = n-1$, then by Lemma \ref{lemma: 2(n-1)} $(a,\frac{b}{n-1})$ is an integer solution of $X^2 - t(n-1)Y^2 = -1$.
\end{itemize}

We now want to show that the numerical conditions in the statement are sufficient to prove the existence of a non-trivial automorphism on $\hsk$. By Theorem \ref{thm_bm_mvb hsk}, from hypotheses $(i)$ and $(ii)$ we deduce that the closure of the movable cone of $\hsk$ is $\overline{\movb(\hsk)} = \langle h, zh-tw\delta \rangle$, with $(z,w)$ as in the statement. Moreover, as we explained in the first part of the proof, hypothesis $(iii)$ guarantees that all classes in the movable cone are ample, i.e.\
\begin{equation}\label{eqt: ample cone} \tag{$\star$}
\mathcal{A}_{\hsk} = \movb(\hsk) = \left\{ xh - y\delta \mid y > 0, zy < twx \right\}.
\end{equation}
\begin{itemize}
\item If the equation $(n-1)X^2 - tY^2 = -1$ admits a solution, let $(a,b)$ be the positive solution with minimal $X$; then, by Lemma \ref{lemma: pell} the minimal solution of Pell's equation $X^2 - t(n-1)Y^2 = 1$ is $(Z, W) = \left(2(n-1)a^2 + 1, 2ab \right)$. Notice, in particular, that $Z \equiv 1$ (mod $n-1$), therefore $(z,w) = (Z,W)$. Moreover, $w = 2ab$ is even and $z \equiv 1$ (mod $2(n-1)$): this implies, as explained in the proof of Proposition \ref{prop: numerical conditions and cones}, that both the extremal rays of the movable cone correspond to divisorial contractions of Hilbert--Chow type. We now set $\nu \coloneqq bh - a\delta \in \ns(\hsk)$, which is a primitive class of square $2$. In particular $\nu$ is ample, using (\ref{eqt: ample cone}), because $a > 0$ and
\[za - twb = a(2(n-1)a^2 + 1 - 2tb^2) = -a < 0.\]
\item If instead there are integer solutions for $X^2 - t(n-1)Y^2 = -1$, let $(a,b)$ be again the minimal one. By Lemma \ref{lemma: pell} the minimal solution of Pell's equation $X^2 - t(n-1)Y^2 = 1$ is $(Z, W) = (2a^2 + 1, 2ab)$. Here $Z \equiv -1$ (mod $n-1$), since $a^2 = t(n-1)b^2 -1$, therefore $(z,w) = (Z, W)$ with $w$ even and $z \equiv -1$ (mod $2(n-1)$). Let $\nu \coloneqq (n-1)bh - a\delta \in \ns(\hsk)$: by Lemma \ref{lemma: 2(n-1)} it is a primitive class of square $2(n-1)$ and divisibility $n-1$. Moreover $\nu$ is ample, using the description (\ref{eqt: ample cone}), because $a > 0$ and
\[za - tw(n-1)b = a(2a^2 + 1 - 2t(n-1)b^2) = -a <0. \]
\end{itemize}

Therefore, if one of the two equations in hypothesis $(iv)$ admits integer solutions, we can construct an ample class $\nu$ with either $(\nu, \nu) = 2$ or $(\nu, \nu) = 2(n-1)$ and $\divi(\nu) = n-1$. By Proposition \ref{prop: existence of involutions} and Remark \ref{remark: invol is non-nat} the Hilbert scheme $\hsk$ admits a non-symplectic, non-natural involution, which acts on $H^2(\hsk, \IZ)$ as $-R_\nu$.
\end{proof}

\begin{rem}
In the proof of Theorem \ref{thm: numerical characterization} we showed that, if condition $(iv)$ holds, then the solution $(z,w)$ of $X^2 - t(n-1)Y^2 = 1$ appearing in the statement is the minimal solution of the equation and $z \equiv \pm 1$ (mod $2(n-1)$). Using formula (\ref{eq: pell solutions}) we deduce that all solutions of Pell's equation $X^2 - 4t(n-1)Y^2 = 1$ have $X \equiv \pm 1$ (mod $2(n-1)$), since $(z, \frac{w}{2})$ is the minimal solution. As a consequence, applying again (\ref{eq: pell solutions}), the congruence class modulo $2(n-1)$ of $X$ is constant (up to sign) for all solutions $(X,Y)$ of $X^2 - 4t(n-1)Y^2 = \alpha^2 - 4\rho(n-1)$ in the same equivalence class (see \S \ref{sec: pell}). Therefore, if there exist solutions $(X,Y)$ of $X^2 - 4t(n-1)Y^2 = \alpha^2 - 4\rho(n-1)$ with $X \equiv \pm \alpha$ (mod $2(n-1)$), there is also a fundamental solution with this property. Moreover, by Remark \ref{rem: fundamental solutions}, the minimal solution $(X,Y)$ of this type has slope $\frac{Y}{X} \leq \frac{w}{2z}$, with $\frac{Y}{X} = \frac{w}{2z}$ if and only if the fundamental solution in the class of $(X,Y)$ is of the form $(h,0)$. In order to verify if hypothesis $(iii)$ of Theorem \ref{thm: numerical characterization} holds, it is therefore sufficient to check that $4t(n-1)Y^2 + \alpha^2 - 4\rho(n-1)$ is not the square of an integer $X \equiv \pm \alpha$ (mod $2(n-1)$) for all integers $Y$ such that $1 \leq Y \leq \sqrt{\frac{(z-1)(\alpha^2 - 4\rho(n-1))}{8t(n-1)}}$ and $\rho, \alpha$ as in the statement of the theorem.
\end{rem}

\begin{rem}
For $n=2$, Theorem \ref{thm: numerical characterization} coincides with \cite[Theorem 1.1]{bcnws}. In fact, the only equation that needs to be considered in condition $(iii)$ is $X^2 - 4tY^2 = 5$, corresponding to $\rho = -1$, $\alpha = 1$. If there exist integer solutions $(X,Y)$ for this equation, they all have $X$ odd, i.e.\ $X \equiv \pm \alpha$ (mod $2(n-1)$), and the minimal solution satisfies $\frac{Y}{X} < \frac{w}{2z}$ (Remark \ref{rem: fundamental solutions}). Condition $(iii)$ of Theorem \ref{thm: numerical characterization} is therefore equivalent to asking that Pell's equation $P_{4t}(5) : X^2 - 4tY^2 = 5$ has no solutions, as requested in \cite[Theorem 1.1]{bcnws}.
\end{rem}

As an application of Theorem \ref{thm: numerical characterization}, it is possible to prove that for any $n \geq 2$ there exist infinite values of $t$ such that $\hsk$ admits a non-symplectic involution $f$ with $T_f \cong \langle 2 \rangle$. In order to do so, we consider a specific sequence of integers $\left\{ t_{n,k} \right\}_{k \in \mathbb{N}}$ and we show that all these $t_{n,k}$'s are admissible if $k$ is sufficiently large.

\begin{prop} \label{prop: infinite cases}
Let $S$ be an algebraic $K3$ surface with $\Pic(S)=\IZ H$, $H^2 = 2t$ and assume $t = (n-1)k^2 + 1$ for $k, n$ positive integers, $n \geq 2$. If $k \geq \frac{n+3}{2}$, there exists a non-symplectic involution $f \in \aut(\hsk)$ with $T_f \cong \langle 2 \rangle$.
\end{prop}

\begin{proof}
We need to verify that the four conditions of Theorem \ref{thm: numerical characterization} are satisfied. If $t = (n-1)k^2 + 1$, for $k \in \mathbb{N}$, it is easy to check that $t(n-1)$ is not a square and that $(X,Y) = (k,1)$ is a solution of Pell's equation $(n-1)X^2 - tY^2 = -1$. 

Moreover, if $n \geq 3$ the equation $(n-1)X^2 - tY^2 = 1$ has no integer solutions. In fact, if the equation was solvable we would be able to find solutions of Pell's equation $X^2 - t(n-1)Y^2 = n-1$ with $X \equiv 0$ (mod $n-1$). Since $t(n-1) > (n-1)^2$, the solutions of this new equation can be determined by looking at the convergents of the continued fraction expansion of $\sqrt{t(n-1)}$ (see for instance \cite[Chapter XXXIII, \S 16]{chrystal}), which is
\[\sqrt{t(n-1)} = \left[ k(n-1); \overline{2k, 2k(n-1)} \right].\]

The evaluation of $X^2 - t(n-1)Y^2$ on the convergents $(X,Y)$ of $\sqrt{t(n-1)}$ is periodic with the same period of the continued fraction, i.e.\ two, and it has alternately values $-n+1$ and $1$. As a consequence, $X^2 - t(n-1)Y^2 = n-1$ has solutions only if $n-1 = h^2$, with $h \in \mathbb{N}$, and in this case the only fundamental solution is $(h,0)$. By using the convergents of $\sqrt{t(n-1)}$ we determine the minimal solution of $X^2 - t(n-1)Y^2 = 1$, which is $(z,w) = (2k^2(n-1)+1, 2k)$. Since $z \equiv 1$ (mod $n-1$), we conclude (using formula (\ref{eq: pell solutions})) that there are no solutions of  $X^2 - t(n-1)Y^2 = n-1$ with $X \equiv 0$ (mod $n-1$). 

It remains to check whether condition $(iii)$ of Theorem \ref{thm: numerical characterization} holds. Assuming $k \geq \frac{n+3}{2}$, we have $4t(n-1) > (\alpha^2 - 4\rho(n-1))^2$ for all $\alpha, \rho$ in the statement of the theorem (notice that $\alpha^2 - 4\rho(n-1)$ can be at most $(n-1)^2 + 4(n-1)$). As before, from this condition we deduce that the solutions of $X^2 - 4t(n-1)Y^2 = \alpha^2 - 4\rho(n-1)$ are encoded in the convergents of the continued fraction of
\[\sqrt{4t(n-1)} = \left[ 2k(n-1); \overline{k, 4k(n-1)} \right].\]

The quadratic form $X^2 - 4t(n-1)Y^2$ takes values $-4(n-1)$ and $1$, respectively, on the first two convergents. As a consequence, if Pell's equation $X^2 - 4t(n-1)Y^2 = \alpha^2 - 4\rho(n-1)$ is solvable, then $\alpha^2 - 4\rho(n-1) = h^2$ for some $h \in \mathbb{N}$ and the minimal solution is $(X,Y) = (hz, h\frac{w}{2})$, with $(z,\frac{w}{2})$ the minimal solution of Pell's equation $X^2 - 4t(n-1)Y^2 = 1$. Since $\frac{Y}{X} = \frac{w}{2z}$,  condition $(iii)$ is satisfied.
\end{proof}

To conclude, we provide a list of the minimal values of $t$ such that there exists a non-natural automorphism $f \in \aut(\hsk)$, for $2 \leq n \leq 12$ (more details on the case $n=2$ can be found in \cite{bcnws}). By Corollary \ref{cor: caso speciale} we have $t \geq 2n-2$. In particular, if $-1$ is a quadratic residue modulo $n-1$ we know by Proposition \ref{prop: generatori T} that the generator of the invariant lattice $T_f$ can either have square $2$ or $2(n-1)$: for these values of $n$ we determine the minimal $t$ in each of the two cases. 

\begin{center}
\begin{tabular}{| >{\centering\arraybackslash}m{0.45in} | >{\centering\arraybackslash}m{1in} | >{\centering\arraybackslash}m{1in} |}
\hline
$n$ & \vspace{0.03in}  \shortstack{$t$ minimum s.t. \\ $T_f \cong \langle 2 \rangle$} & \vspace{0.03in} \shortstack{$t$ minimum s.t. \\ $T_f \cong \langle 2(n-1) \rangle$}\\
\hline
$2$ & \multicolumn{2}{c|}{$2$} \\\hline
$3$ & $19$ & $13$  \\\hline
$4$ & $19$ & /  \\\hline
$5$ & $37$ & /  \\\hline
$6$ & $46$ & $34$  \\\hline
$7$ & $55$ & /  \\\hline
$8$ & $64$ & /  \\\hline
$9$ & $73$ & /  \\\hline
$10$ & $82$ & / \\\hline
$11$ & $91$ & $73$ \\\hline
$12$ & $100$ & / \\\hline
\end{tabular}
\end{center}

For $n=3$, in \cite[Example 14]{hassett_tschinkel} the authors proved the existence of a non-natural automorphism $f \in \aut(S^{[3]})$ with $T_f \cong \langle 2 \rangle$ when $S$ is a $K3$ surface with $\Pic(S) = \IZ H$, $H ^ 2 = 114$ (i.e.\ $t=57$). However, as shown in the table above, this is not the minimal $t$ for which a similar automorphism exists on $S^{[3]}$.

We notice an interesting pattern in the second column of the table: except for $n=2$ and $n=4$, the minimum value of $t$ such that there exists an automorphism $f \in \aut(\hsk)$ with invariant lattice $T_f \cong \langle 2 \rangle$ is always of the form $t = 9n-8$. We point out that this is one of the values of $t$ considered in Proposition \ref{prop: infinite cases}, corresponding to $k = 3$. This suggests that the lower bound on $k$ provided in Proposition \ref{prop: infinite cases} for the existence of a non-natural involution may be improved.

\bibliographystyle{amsplain}
\bibliography{AutomorphismGroupK3n}
\end{document}